\documentclass[amstex,12pt,reqno]{amsart}
\usepackage{amsmath, amssymb, amsthm}
\usepackage{mathrsfs}

\newtheorem{theorem}{Theorem}
\newtheorem{lemma}[theorem]{Lemma}
\newtheorem{proposition}[theorem]{Proposition}
\newtheorem{corollary}[theorem]{Corollary}

\theoremstyle{definition}

\newtheorem{remark}{Remark}

\title[Teichm\"uller space of 
circle diffeomorphisms]{The complex structure of the Teichm\"uller space of 
circle diffeomorphisms in the Zygmund smooth class}
\author[K. Matsuzaki]{Katsuhiko Matsuzaki}
\address{Department of Mathematics, School of Education, Waseda University \endgraf
Shinjuku, Tokyo 169-8050, Japan}
\email{matsuzak@waseda.jp}

\makeatletter
\@namedef{subjclassname@2020}{%
\textup{2020} Mathematics Subject Classification}
\makeatother

\thanks{Research supported by 
Japan Society for the Promotion of Science (KAKENHI 23H01078 and 23K17656)}

\begin{document}

\maketitle

\begin{abstract}
We provide the complex Banach manifold structure for the Teichm\"uller space of 
circle diffeomorphisms whose derivatives are in the Zygmund class. This is done by
showing that the Schwarzian derivative map is a holomorphic split submersion.
\end{abstract}

\section{Introduction}

The groups of circle diffeomorphisms of certain regularity have been studied in the framework of
the theory of the universal Teichm\"uller space. In the previous papers \cite{M1} and \cite{TW2}, 
we considered the class of
circle diffeomorphisms whose derivatives are $\gamma$-H\"older continuous for $0<\gamma<1$, and gave the foundation of
this Teichm\"uller space. In the limiting case of
this $\gamma$-H\"older continuity as $\gamma \to 1$, one can adopt the Zygmund continuous condition.
The corresponding Teichm\"uller space $T_Z$ of those circle diffeomorphisms has been defined in Tang and Wu \cite{TW1}.
In this present paper, we endow $T_Z$ with the complex Banach mani\-fold structure and prove several important properties of $T_Z$
regarding this complex structure.

We denote by
${\rm Bel}(\mathbb D^*)$ the space of Beltrami coefficients $\mu$, which are measurable functions on the exterior of the unit disk $\mathbb D^*=\{z \mid |z|>1\} \cup \{\infty\}$ with the $L^\infty$ norm $\Vert \mu \Vert_\infty$ less than $1$. To investigate the Zygmund smooth class, we use the space
$$
{\rm Bel}_Z(\mathbb D^*)=\{ \mu \in {\rm Bel}(\mathbb D^*) \mid \Vert \mu \Vert_Z={\rm ess\, sup}_{|z|>1}((|z|^2-1)^{-1}\lor 1)|\mu(z)|<\infty\}.
$$
It is said that $\mu$ and $\nu$ in ${\rm Bel}(\mathbb D^*)$ are {\it Teichm\"uller equivalent} if the normalized quasiconformal self-homeomorphisms $f^\mu$ and $f^\nu$ of
$\mathbb D^*$ having $\mu$ and $\nu$ as their complex dilatations coincide on the unit circle $\mathbb S=\partial \mathbb D^*$.
The {\it universal Teichm\"uller space} $T$ is the quotient space of ${\rm Bel}(\mathbb D^*)$ by the Teichm\"uller equivalence. 
We denote this projection by $\pi:{\rm Bel}(\mathbb D^*) \to T$. Similarly,
the Teichm\"uller space $T_Z$ is defined to be the quotient space of ${\rm Bel}_Z(\mathbb D^*)$ by the Teichm\"uller equivalence. 
This can be identified with the set of all normalized orientation-preserving self-diffeomorphisms $f$ of $\mathbb S$ whose derivatives $f'$ are continuous and satisfy the following {\it Zygmund condition}: there exists some constant $C>0$ such that
$$
|f'(e^{i(\theta+t)})-2f'(e^{i\theta})+f'(e^{i(\theta-t)})| \leq Ct
$$
for all $\theta \in [0,2\pi)$ and $t>0$. See \cite[Theorem 1.1]{TW1} for these characterizations of $T_Z$.

To introduce the complex structures to $T$,
we prepare the complex Banach space $A(\mathbb D)$ of holomorphic functions $\varphi$ on the unit disk $\mathbb D=\{z \mid |z|<1\}$ with 
$$
\Vert \varphi \Vert_A={\rm sup}_{|z|<1}(1-|z|^2)^2|\varphi(z)|<\infty.
$$
In the case of $T_Z$, we consider the complex Banach space
$$
A_Z(\mathbb D)=\{\varphi \in A(\mathbb D) \mid \Vert \varphi \Vert_{A_Z}={\rm sup}_{|z|<1}(1-|z|^2)|\varphi(z)|<\infty\}
$$
as the corresponding space.
It is obvious that $\Vert \varphi \Vert_A \leq \Vert \varphi \Vert_{A_Z}$.

For every $\mu \in {\rm Bel}(\mathbb D^*)$, let $f_\mu$ be the quasiconformal homeomorphism of the extended complex plane $\widehat {\mathbb C}$
with complex dilatation $0$ on $\mathbb D$ and $\mu$ on $\mathbb D^*$. Then, the {\it Schwarzian derivative map} $\Phi$ is
defined by the correspondence of $\mu$ to $S_{f_\mu|_{\mathbb D}}$, where for a locally univalent 
holomorphic function $f$ in general, 
the Schwarzian derivative of $f$ is defined by 
$$
S_f(z)=\left(\frac{f''(z)}{f'(z)}\right)'-\frac{1}{2}\left(\frac{f''(z)}{f'(z)}\right)^2.
$$
It is known that $S_{f_\mu|_{\mathbb D}}$ belongs to $A(\mathbb D)$ and $\Phi: {\rm Bel}(\mathbb D^*) \to A(\mathbb D)$ is a holomorphic split submersion (see \cite[Section 3.4]{Na}).
Moreover, $\Phi$ projects down to a well-defined injection $\alpha:T \to A(\mathbb D)$ such that $\alpha \circ \pi=\Phi$. This is called the {\it Bers embedding} of $T$.
By the property of $\Phi$, we see that $\alpha$ is a homeomorphism onto the image. This provides $T$ with the complex Banach structure of $A(\mathbb D)$, 
which is the unique complex structure on $T$  such that the Teichm\"uller projection $\pi: {\rm Bel}(\mathbb D^*) \to T$ is a holomorphic map
with surjective derivatives
at all points of ${\rm Bel}(\mathbb D^*)$.

For the Teichm\"uller space $T_Z$, it was proved in \cite[Theorem 2.8]{TW1} that 
$$
\Phi({\rm Bel}_Z(\mathbb D^*))=A_Z(\mathbb D) \cap \Phi({\rm Bel}(\mathbb D^*)).
$$
Moreover, 
the following theorem was also proved.

\begin{theorem}[\mbox{\cite[Theorem 1.2]{TW1}}]\label{holo}
The Schwarzian derivative map $\Phi:{\rm Bel}_Z(\mathbb D^*) \to A_{Z}(\mathbb D)$ is holomorphic.
\end{theorem}

\begin{corollary}\label{betaconti}
The Bers embedding $\alpha:T_Z \to A_{Z}(\mathbb D)$ is continuous.
\end{corollary}

Indeed, as the Teichm\"uller projection $\pi:{\rm Bel}_Z(\mathbb D^*) \to T_Z$ is a quotient map
satisfying $\alpha \circ \pi=\Phi$, the continuity of $\alpha$ is equivalent to that of $\Phi$.

In this paper, these results are improved as is the case for T and many other Teichm\"uller spaces.

\begin{theorem}\label{main}
The following hold true:
\begin{enumerate}
\item
The Schwarzian derivative map $\Phi:{\rm Bel}_Z(\mathbb D^*) \to A_{Z}(\mathbb D)$ is a holomorphic split submersion
onto the image.
\item
The Bers embedding $\alpha:T_Z \to A_Z(\mathbb D)$ is a homeomorphism onto 
the image $\alpha(T_Z)=\Phi({\rm Bel}_Z(\mathbb D^*))$,
which is a connected open subset of $A_Z(\mathbb D)$.
\item
The Teichm\"uller space $T_Z$ is endowed with the unique complex Banach manifold structure such that
the Teichm\"uller projection $\pi:{\rm Bel}_Z(\mathbb D^*) \to T_Z$ is a holomorphic map with surjective derivatives
at all points of ${\rm Bel}_Z(\mathbb D^*)$.
\end{enumerate}
\end{theorem}

These claims are shown in Corollaries \ref{holosub}, \ref{homeo}, and \ref{complex}, respectively,
as the consequences from Theorem \ref{submersion} in Section \ref{2}, which supplies the property of split submersion
to the holomorphic map $\Phi$ given in Theorem \ref{holo}.

The Teichm\"uller space $T_Z$ can be also embedded into the space of pre-Schwarzian derivatives of $f_\mu|_{\mathbb D}$ instead of using
Schwarzian derivatives. This was also investigated in \cite{TW1} and it was proved that this correspondence to $\mu$ is holomorphic.
In Section \ref{3}, we also improve the results for this model of $T_Z$.

\section{The complex Banach structure}\label{2}

In this section, we prove Theorem \ref{main}. This will be done by the combination of Theorem \ref{holo} and
Theorem \ref{submersion}, which is the main achievement in this paper. In the first half of this section, we
show necessary claims towards this goal. 

For $\mu$ and $\nu$ in ${\rm Bel}(\mathbb D^*)$, we denote by $\mu \ast \nu$
the complex dilatation of $f^\mu \circ f^\nu$, and by $\nu^{-1}$ the complex dilatation of $(f^\nu)^{-1}$.

\begin{proposition}\label{composition}
If $\mu, \nu \in {\rm Bel}_Z(\mathbb D^*)$, then $\mu \ast \nu^{-1} \in {\rm Bel}_Z(\mathbb D^*)$. Moreover,
for every $\nu \in {\rm Bel}_Z(\mathbb D^*)$, the right translation
$r_\nu:{\rm Bel}_Z(\mathbb D^*) \to {\rm Bel}_Z(\mathbb D^*)$ defined by $\mu \mapsto \mu \ast \nu^{-1}$
is continuous.
\end{proposition}

\begin{proof}
By the formula of the complex dilatation of the composition, we have
$$
\left |(\mu \ast \nu^{-1})(f^\nu(z))\right|=\left| \frac{\mu(z)-\nu(z)}{1-\overline{\nu(z)}\mu(z)}\right| \leq C|\mu(z)-\nu(z)|
$$
for a constant $C>0$ depending only on $\Vert \nu \Vert_\infty$.
This implies that $(\mu \ast \nu^{-1})\circ f^\nu \in {\rm Bel}_Z(\mathbb D)$.
By \cite[Theorem 6.4]{M1}, we see that $1-|f^\nu(z)|^2 \asymp 1-|z|^2$ (the symbol $\asymp$ stands for the equality modulo
a uniform positive constant multiple), where the multiple constant
depends only on $\nu$. Hence,
$$
\frac{|(\mu \ast \nu^{-1})(w)|}{1-|w|^2}=\frac{|(\mu \ast \nu^{-1})(f^\nu(z))|}{1-|f^\nu(z)|^2}\asymp \frac{|(\mu \ast \nu^{-1})\circ f^\nu(z)|}{1-|z|^2}
$$
for $w=f^\nu(z)$.
This shows that $\mu \ast \nu^{-1} \in {\rm Bel}_Z(\mathbb D)$.

Similarly, for any $\mu_1, \mu_2 \in {\rm Bel}_Z(\mathbb D^*)$, we have
\begin{align*}
|(r_\nu(\mu_1)-r_\nu(\mu_2))(f^\nu(z))| &= \frac{|\mu_1(z)-\mu_2(z)|(1-|\nu(z)|^2)}{|1-\overline{\nu(z)}\mu_1(z)||1-\overline{\nu(z)}\mu_2(z)|}\\
&\leq \frac{|\mu_1(z)-\mu_2(z)|}{(1-|\mu_1(z)|^2)^{1/2}(1-|\mu_2(z)|^2)^{1/2}}.
\end{align*}
This shows that the right translation $r_\nu$ is continuous for every $\nu \in  {\rm Bel}_Z(\mathbb D^*)$.
\end{proof}

The continuity is usually promoted to the holomorphy in the arguments of subspaces of the universal Teichm\"uller space.

\begin{lemma}\label{biholo}
For every $\nu \in {\rm Bel}_Z(\mathbb D^*)$, the right translation $r_\nu$ is a biholomorphic automorphism of ${\rm Bel}_Z(\mathbb D^*)$.
\end{lemma}

\begin{proof}
This is a standard consequence from Proposition \ref{composition}. We know that $r_\nu:{\rm Bel}(\mathbb D^*) \to {\rm Bel}(\mathbb D^*)$ 
is holomorphic in the norm of ${\rm Bel}(\mathbb D^*)$. Then, the continuity of $r_\nu$ in the norm of ${\rm Bel}_Z(\mathbb D^*)$
implies the holomorphy. See \cite[p.206]{Le}. Proposition \ref{composition} also implies that $r_{\nu^{-1}}$ is continuous.
As $r_{\nu^{-1}}=(r_{\nu})^{-1}$, we see that $(r_{\nu})^{-1}$ is also holomorphic, and that $r_\nu$ is biholomorphic.
\end{proof}

\begin{remark}
The right translation $r_\nu$ for every $\nu \in  {\rm Bel}_Z(\mathbb D^*)$ is projected down under $\pi: {\rm Bel}_Z(\mathbb D^*) \to T_Z$ to a well-defined bijection
$R_{[\nu]}:T_Z \to T_Z$ which depends only on the Teichm\"uller equivalence class $[\nu] \in T_Z$. After obtaining Theorem \ref{submersion} and Corollary \ref{complex} below, we see that
$R_{[\nu]}$ is a biholomorphic automorphism of $T_Z$.
\end{remark}

It is important to choose a suitable representative $\nu$ in each Teichm\"uller equivalence class. 
Practically, the bi-Lipschitz condition on $f^\nu$ is often required.

\begin{lemma}\label{class}
For every $\mu \in {\rm Bel}_Z(\mathbb D^*)$, there exists $\nu \in {\rm Bel}_Z(\mathbb D^*)$ such that
$\nu$ is Teichm\"uller equivalent to $\mu$ and $f^{\nu}$ is a bi-Lipschitz real-analytic diffeomorphism in the hyperbolic metric 
on $\mathbb D^*$.
\end{lemma}

\begin{proof}
Suppose that $\Vert \mu \Vert_A<1/3$. Then, $\varphi=\Phi(\mu)$ satisfies $\Vert \varphi \Vert_A <2$ (see \cite[Section II.3.3]{Le}).
The Ahlfors--Weill section $\sigma$ for the Schwarzian derivative map $\Phi$ is defined by
$$
\sigma(\varphi)(z^*)=-\frac{1}{2}(zz^*)(1-|z|^2)^2\varphi(z)
$$
for $z^*=1/\bar z \in \mathbb D^*$ (see \cite[Section II.5.1]{Le}). Let $\nu=\sigma(\varphi)$. Then, $\nu$ is Teichm\"uller equivalent to $\mu$ due to $\Phi \circ \sigma={\rm id}$. Moreover,
$\nu$ belongs to ${\rm Bel}_Z(\mathbb D^*)$ because $\varphi \in A_Z(\mathbb D)$ by Theorem \ref{holo}, and $f^\nu$ is a bi-Lipschitz 
real-analytic self-diffeomorphism of $\mathbb D^*$
in the hyperbolic metric
by \cite[Theorem 8]{M2}. A similar claim can be found in \cite[p.27]{TT}.

For an arbitrary $\mu \in {\rm Bel}(\mathbb D^*)$, by choosing $n \in \mathbb N$ so that
$n \geq 3\Vert \mu \Vert_\infty/(1-\Vert \mu \Vert_\infty)$, we set
$\mu_k=k\mu/n$ $(k=0,1,\ldots,n)$. Then, 
$\Vert \mu_{k+1} \ast \mu_k^{-1} \Vert_\infty<1/3$ is satisfied.
We will prove the desired claim by induction.
Suppose that we have obtained $\nu_k \in {\rm Bel}_Z(\mathbb D^*)$ such that $\nu_k$ is Teichm\"uller equivalent to $\mu_k$ 
and $f^{\nu_k}$ is
bi-Lipschitz. By Proposition \ref{composition}, $\mu_{k+1} \ast \nu_k^{-1}$ is in ${\rm Bel}_Z(\mathbb D^*)$, which 
is Teichm\"uller equivalent to $\mu_{k+1} \ast \mu_k^{-1}$. Hence, 
$$
\Phi(\mu_{k+1} \ast \nu_k^{-1})=\Phi(\mu_{k+1} \ast \mu_k^{-1}) \in A_Z(\mathbb D), 
$$
and
the argument in the first part of this proof implies that
$\sigma(\Phi(\mu_{k+1} \ast \nu_k^{-1})) \in {\rm Bel}_Z(\mathbb D^*)$ and this yields a bi-Lipschitz self-diffeomorphism of $\mathbb D^*$.
Then, its composition with $f^{\nu_k}$ is also a bi-Lipschitz diffeomorphism whose complex dilatation is defined to be $\nu_{k+1}$.
Namely, 
$$
\nu_{k+1}=\sigma(\Phi(\mu_{k+1} \ast \nu_k^{-1})) \ast \nu_k=\sigma(\Phi(\mu_{k+1} \ast \nu_k^{-1})) \ast (\nu_k^{-1})^{-1}. 
$$
This is Teichm\"uller equivalent to $\mu_{k+1}$ and belongs to ${\rm Bel}_Z(\mathbb D^*)$ by Proposition \ref{composition}.
Thus, the induction step proceeds. We obtain $\nu=\nu_n$ as a required replacement of $\mu$.
\end{proof}

\begin{remark} For circle diffeomorphisms whose derivatives are 
$\gamma$-H\"older continuous for $0 < \gamma < 1$, the conformally
barycentric extension (the Douady--Earle extension) yields an appropriate bi-Lipschitz diffeomorphism
(see \cite[Theorem 6.10]{M1}), but we do not know whether this works also for the Zygmund smooth class. The
Beurling--Ahlfors extension of self-diffeomorphisms of R whose derivatives are in the Zygmund smooth class
is investigated in \cite{Hu}.
\end{remark}

We are ready to state the main claim in our arguments.

\begin{theorem}\label{submersion}
For every $\mu \in {\rm Bel}_Z(\mathbb D^*)$, there exists a holomorphic map $s_{\mu}:V_\varphi \to {\rm Bel}_Z(\mathbb D^*)$ defined on some 
neighborhood $V_\varphi$ of $\varphi=\Phi(\mu)$ in $A_Z(\mathbb D)$ such that $s_\mu(\varphi)=\mu$ and $\Phi \circ s_{\mu}$ is the identity on $V_\varphi$.
\end{theorem}

\begin{proof}
The existence of a local holomorphic right inverse $s_\nu$ of $\Phi$ such that $s_\nu \circ \Phi(\nu)=\nu$ can be proved by the standard argument if
$f^\nu$ is bi-Lipschitz in the hyperbolic metric for $\nu \in {\rm Bel}_Z(\mathbb D^*)$. This is obtained as the generalized Ahlfors--Weill section
by using the bi-Lipschitz quasiconformal reflection with respect to the quasicircle $f_\nu(\mathbb S)$. See \cite{EN} for a general argument, and \cite{M1,WM} for its application to particular
Teichm\"uller spaces. In our present case, the same proof can be applied.

For an arbitrary $\mu \in {\rm Bel}_Z(\mathbb D^*)$, Lemma \ref{class} shows that there exists $\nu \in {\rm Bel}_Z(\mathbb D^*)$ such that
$\nu$ is Teichm\"uller equivalent to $\mu$ and $f^{\nu}$ is bi-Lipschitz in the hyperbolic metric.
Then, we can take a local holomorphic right inverse $s_\nu$ defined on some neighborhood $V_\varphi$ of $\varphi=\Phi(\nu)$ in $A_Z(\mathbb D)$
such that $s_\nu \circ \Phi(\mu)=s_\nu \circ \Phi(\nu)=\nu$.
By using the right translations $r_\mu$ and $r_\nu$ which are biholomorphic automorphisms of  
${\rm Bel}_Z(\mathbb D^*)$ as in Lemma \ref{biholo}, 
we set $s_\mu=r_{\mu}^{-1} \circ r_{\nu} \circ s_\nu$. This is a local holomorphic right inverse of $\Phi$ defined on $V_\varphi$ 
satisfying $s_\mu \circ \Phi(\mu)=\mu$.
\end{proof}

The statements of Theorem \ref{main} correspond to the following corollaries to Theorem \ref{submersion}.

\begin{corollary}\label{holosub}
The Schwarzian derivative map $\Phi:{\rm Bel}_Z(\mathbb D^*) \to A_{Z}(\mathbb D)$ is a holomorphic split submersion.
\end{corollary}

\begin{proof}
By Theorem \ref{holo}, $\Phi$ is holomorphic, and by Theorem \ref{submersion}, it is a holomorphic split submersion.
\end{proof}

\begin{remark}\label{rem2}
For a surjective holomorphic map $\Phi:B \to A$ from a domain $B \subset Y$ to a domain $A \subset X$ of Banach spaces $X$ and $Y$ in general,
$\Phi$ is a holomorphic split submersion if and only if both of the following conditions are satisfied (see \cite[p.89]{Na}):
\begin{enumerate}
\item
For every $\varphi \in A$, there exists a holomorphic map $s$ defined on a neighborhood $V \subset A$ of $\varphi$ such that
$\Phi \circ s={\rm id}_V$;
\item
Every $\mu \in B$ is contained in the image of some local holomorphic right inverse $s$ as given in (1).
\end{enumerate}
For certain Teichm\"uller spaces defined by the supremum norm such as
the universal Teichm\"uller space and Teichm\"uller spaces of circle diffeomorphisms,
it has been proved that the Schwarzian derivative map $\Phi$ is a holomorphic split submersion. However,
for Teichm\"uller spaces defined by the integrable norm, 
only condition (1) has been verified.  
\end{remark}

\begin{corollary}\label{homeo}
The Bers embedding $\alpha:T_Z \to A_Z(\mathbb D)$ is a homeomorphism onto 
the image $\alpha(T_Z)=\Phi({\rm Bel}_Z(\mathbb D^*))$,
which is a connected open subset of $A_Z(\mathbb D)$.
\end{corollary}

\begin{proof}
The image $\alpha(T_Z)=\Phi({\rm Bel}_Z(\mathbb D^*))$ is open because for every $\varphi \in \Phi({\rm Bel}_Z(\mathbb D^*))$,
the neighborhood $V_\varphi$ as in Theorem \ref{submersion} is contained in $\Phi({\rm Bel}_Z(\mathbb D^*))$.
By Corollary \ref{betaconti}, the Bers embedding $\alpha:T_Z \to \alpha(T_Z)$ is continuous. Conversely, since 
there is a local holomorphic right inverse $s_\mu$ of $\Phi$ defined on some neighborhood $V_\varphi$ for every $\varphi \in \alpha(T_Z)$
satisfying
$\alpha^{-1}|_{V_\varphi}=\pi \circ s_\mu$ with $\Phi(\mu)=\varphi$, we see that $\alpha^{-1}$ is continuous.
\end{proof}

\begin{corollary}\label{complex}
The Teichm\"uller space $T_Z$ is endowed with the unique complex Banach manifold structure such that
the Teichm\"uller projection $\pi:{\rm Bel}_Z(\mathbb D^*) \to T_Z$ is a holomorphic map with surjective derivatives
at all points of ${\rm Bel}_Z(\mathbb D^*)$.
\end{corollary}

\begin{proof}
As the Bers embedding $\alpha:T_Z \to A_Z(\mathbb D)$ is a homeomorphism onto the image by Corollary \ref{homeo},
$T_Z$ is endowed with the complex Banach structure as the domain $\alpha(T_Z)$ in $A_Z(\mathbb D)$.
The uniqueness follows from \cite[Proposition 8]{EGL}.
\end{proof}

\section{Pre-Schwarzian derivatives}\label{3}

Let $\psi_f(z)=\log f'(z)$ for a locally univalent holomorphic function $f$ on $\mathbb D$. 
Its derivative $P_f(z)=\psi_f'(z)$ is called the pre-Schwarzian derivative of $f$.
Let $B(\mathbb D)$ be the space of all holomorphic functions $\psi$ on $\mathbb D$ such that 
$$
\Vert \psi \Vert_B={\rm sup}_{|z|<1}(1-|z|^2)|\psi'(z)|<\infty. 
$$
Such a $\psi$ is called a {\it Bloch function}. Then, ignoring the difference of complex constant functions, we can regard
$B(\mathbb D)$ as the complex Banach space with norm $\Vert \cdot \Vert_B$.
We note that $\psi \in B(\mathbb D)$ if and only if $\psi' \in A_Z(\mathbb D)$.

We require that $f_\mu$ satisfies $f_\mu(\infty)=\infty$.
Then, $P_{{f_\mu}|_{\mathbb D}}$ is well defined for every $\mu \in {\rm Bel}(\mathbb D^*)$ by this normalization and 
$\psi_{{f_\mu}|_{\mathbb D}}$ belongs to $B(\mathbb D)$. We may assume that $\psi_{{f_\mu}|_{\mathbb D}}(0)=0$.
By this correspondence $\mu \mapsto \psi_{{f_\mu}|_{\mathbb D}}$, we have a holomorphic map $\Psi:{\rm Bel}(\mathbb D^*) \to B(\mathbb D)$ called the
{\it pre-Schwarzian derivative map}.

Let $\widetilde {\mathcal T}=\Psi({\rm Bel}(\mathbb D^*))$, which consists of those $\psi_f$ for 
a conformal homeomorphism $f$ of $\mathbb D$ that is quasiconformally extendable to $\mathbb D^*$
with $f(\mathbb D)$ bounded. 
It is known that this is an open subset of $B(\mathbb D)$ but there are other uncountably many
connected components of the open subset in $B(\mathbb D)$
consisting of those $\psi_{f}$ 
with $f(\mathbb D)$ unbounded. See \cite{Z}.

In the case of the Teichm\"uller space $T_Z$, we define the corresponding space
$$
B_Z(\mathbb D)=\{ \psi \in B(\mathbb D) \mid \Vert \psi \Vert_{B_Z}=|\psi'(0)|+{\rm sup}_{|z|<1}(1-|z|^2)|\psi''(z)|<\infty\},
$$
which we regard as the complex Banach space by ignoring the difference of complex constant functions.
We note that $\psi \in B_Z(\mathbb D)$ if and only if $\psi' \in B(\mathbb D)$. Moreover,
we see that $\Vert \psi \Vert_B \lesssim \Vert \psi \Vert_{B_Z}$ (see \cite[Theorem 5.4]{Zh}).
Hereafter, the symbol $\lesssim$ stands for the inequality modulo a uniform positive constant multiple.
The space $B_Z$ was defined in \cite{TW1} by using $\sup_{z \neq w \in \mathbb D}|\psi'(z)-\psi'(w)|/d(z,w)$ for the hyperbolic distance $d$ on $\mathbb D$, but
they are equivalent (see \cite[Theorem 5.5]{Zh}).

Concerning the pre-Schwarzian derivative model of $T_Z$,
the following theorem has been obtained:

\begin{theorem}[\mbox{\cite[Theorem 1.4]{TW1}}]\label{P-holo}
The image of ${\rm Bel}_Z(\mathbb D^*)$ under the pre-Schwarzian derivative map $\Psi$
is contained in $B_Z(\mathbb D)$ and $\Psi:{\rm Bel}_Z(\mathbb D^*) \to B_Z(\mathbb D)$
is holomorphic.
\end{theorem}

We set this image as $\widetilde {\mathcal T}_Z=\Psi({\rm Bel}_Z(\mathbb D^*)) \subset B_Z(\mathbb D)$.
It can be seen that 
$$
\widetilde {\mathcal T}_Z=B_Z(\mathbb D) \cap \widetilde {\mathcal T}
$$ 
by \cite[Theorem 1.1]{TW1}.
This is an open subset of $B_Z(\mathbb D)$ because $\widetilde {\mathcal T}$ is open in $B(\mathbb D)$ and 
$\Vert \cdot \Vert_B \lesssim \Vert \cdot \Vert_{B_Z}$. Moreover, differently from the case of the universal Teichm\"uller space $T$,
it is proved in \cite[Theorem 1.3]{TW1} that $\widetilde {\mathcal T}_Z$ is exactly the set of those $\psi_{f}$ in $B_Z(\mathbb D)$
for a conformal homeomorphism $f$ of $\mathbb D$ that is quasiconformally extendable to $\mathbb D^*$ with its complex dilatation in 
${\rm Bel}_Z(\mathbb D^*)$.

We examine the structure of $\widetilde {\mathcal T}_Z$ and the
pre-Schwarzian derivative map $\Psi:{\rm Bel}_Z(\mathbb D^*) \to \widetilde {\mathcal T}_Z$. The strategy is to factorize the Schwarzian derivative map $\Phi$ by $\Psi$ and
to bring the properties of $\Phi$ to $\Psi$. Due to the relation $S_f=(P_f)'-(P_f)^2/2$, we consider the following map $\Lambda$
satisfying $\Lambda \circ \Psi=\Phi$. 

\begin{lemma}\label{BtoA}
The map $\Lambda:B_Z(\mathbb D) \to A_Z(\mathbb D)$ defined by $\Lambda(\psi)=\psi''-\frac{1}{2}(\psi')^2$ is holomorphic.
\end{lemma}

\begin{proof}
It is easy to see that $\Lambda$ is G\^ateaux holomorphic. 
Hence, to prove that $\Lambda$ is holomorphic,
it suffices to show that $\Lambda$ is locally bounded. See \cite[p.28]{Bour}. The local boundedness can be verified as follows.

Let $\psi \in B_Z(\mathbb D)$. Then $\psi'' \in A_Z(\mathbb D)$ with $\Vert \psi'' \Vert_{A_Z} \leq \Vert \psi \Vert_{B_Z}$ by definition.
Moreover, for any $\phi \in B_Z(\mathbb D)$, we have $\phi'\psi' \in A_Z(\mathbb D)$ with 
$\Vert \phi'\psi' \Vert_{A_Z} \lesssim \Vert \phi \Vert_{B_Z}\Vert \psi \Vert_{B_Z}$,
and in particular, $\Vert (\psi')^2 \Vert_{A_Z} \lesssim \Vert \psi \Vert_{B_Z}^2$. 
Indeed,
$$
\Vert \phi'\psi' \Vert_{A_Z}=\sup_{|z|<1}(1-|z|^2)|\phi'(z)||\psi'(z)| \leq \sup_{|z|<1}(1-|z|^2)^{1/2}|\phi'(z)| 
\cdot \sup_{|z|<1}(1-|z|^2)^{1/2}|\psi'(z)|.
$$
Here, by \cite[Proposition 8]{Zh1}, we see that
$$
\sup_{|z|<1}(1-|z|^2)^{1/2}|\psi'(z)| \asymp |\psi'(0)|+\sup_{|z|<1}(1-|z|^2)^{3/2}|\psi''(z)| \leq \Vert \psi \Vert_{B_Z}.
$$
Thus, $\Vert \phi'\psi' \Vert_{A_Z}$ is bounded as required. Therefore,
$$
\Vert \Lambda(\psi) \Vert_{A_Z} \leq \Vert \psi'' \Vert_{A_Z}+\frac{1}{2}\Vert (\psi')^2 \Vert_{A_Z} \lesssim 
\Vert \psi \Vert_{B_Z}+\frac{1}{2}\Vert \psi \Vert_{B_Z}^2,
$$
which implies that $\Lambda$ is locally bounded.
\end{proof}

We state the main claim in this section, which improves
Theorem \ref{P-holo}.

\begin{theorem}\label{final}
Both $\Psi:{\rm Bel}_Z(\mathbb D^*) \to \widetilde {\mathcal T}_Z \subset B_Z(\mathbb D)$ and
$\Lambda|_{\widetilde {\mathcal T}_Z}:\widetilde {\mathcal T}_Z \to \alpha(T_Z) \subset A_Z(\mathbb D)$ 
are holomorphic split submersions. 
\end{theorem}

\begin{proof}
By Theorem \ref{submersion}, $\Phi=\Lambda \circ \Psi$ is a holomorphic split submersion. Combined with Theorem \ref{P-holo} and
Lemma \ref{BtoA},
this implies the statement.
\end{proof}

In fact, $\widetilde {\mathcal T}_Z$ is a disk-bundle over $T_Z \cong \alpha(T_Z)$.

\begin{remark}
For the Teichm\"uller space of circle diffeomorphisms of $\gamma$-H\"older continuous derivatives,
the fact that the pre-Schwarzian derivative map $\Psi$ is a holomorphic split submersion 
was proved in \cite[Theorem 1.2]{TW2}.
Once we know that $\Psi$ is holomorphic (in fact, it suffices to prove its continuity to see this), 
the fact that the Schwarzian derivative map $\Phi$ is a holomorphic split submersion
shown in \cite[Theorem 7.6]{M1} implies that so is $\Psi$ by similar arguments as above.
\end{remark}

\end{document}